\documentclass{article}
\usepackage{amsmath,amsfonts,amssymb,amsthm,enumerate}
\usepackage[all]{xy}
\newtheorem{proposition}{Proposition}[section]
\newtheorem{example}[proposition]{Example}
\newcommand{\C}{\mathbb{C}}
\newcommand{\X}{\mathbb{X}}
\newcommand{\SE}[1]{\mathbf{SplExt}(#1)}
\newcommand{\K}{K}
\newcommand{\Pro}{P}
\newcommand{\Z}{\mathbb{Z}}
\newcommand{\Set}{\mathbf{Set}}
\newcommand{\SplExt}{\textnormal{SplExt}}
\newcommand{\Aut}{\textnormal{Aut}}
\newcommand{\op}{\textnormal{op}}

\begin{document}
\title{A note on the relationship between action accessible and weakly action representable categories}
\author{James Richard Andrew Gray}
\maketitle
\begin{abstract}
The main purpose of this paper is to show that the converse of the known
implication \emph{weakly action representable implies action accessible} is false.
In particular we show that both action accessibility, as well as the (at
least formally stronger) condition requiring the existence of all normalizers
do not imply weakly-action-representability even for varieties.
In addition we show that in contrast to both action accessibility
and the condition requiring the existence of all normalizers, weakly-action
representability is not necessarily inherited by Birkoff subcategories.
\end{abstract}
\section{Introduction}
Recall that for a pointed category $\C$, a split extension (of $B$ with kernel
 $X$) is a diagram in $\C$
\begin{equation}
\vcenter{
\label{diagram: split extension}
\xymatrix{X \ar[r]^{\kappa} & A \ar@<0.5ex>[r]^{\alpha} & B\ar@<0.5ex>[l]^{\beta}}
}
\end{equation}
where $\kappa$ is the kernel of $\alpha$, and $\alpha\beta = 1_B$.  A morphism of split extensions is a diagram in $\C$
\begin{equation}
\label{diagram: morphism of split extensions}
\vcenter{
\xymatrix{
X \ar[r]^{\kappa} \ar[d]_{u}& A \ar@<0.5ex>[r]^{\alpha}\ar[d]^{v} & B\ar@<0.5ex>[l]^{\beta}\ar[d]^{w}\\
X' \ar[r]^{\kappa'} & A' \ar@<0.5ex>[r]^{\alpha'} & B'\ar@<0.5ex>[l]^{\beta'}
}
}
\end{equation}
where the top and bottom rows are split extensions (the domain and codomain respectively), and $v\kappa = \kappa' u$, $v\beta = \beta' w$ and $w\alpha=\alpha'v$. 
Let us denote by $\SE{\C}$ the category of split extensions in $\C$, and by $\K$ and $\Pro$ the functors sending a split extension to its kernel and codomain, respectively. These data together form a span
\begin{equation}
\label{points span i}
\vcenter{
\xymatrix{
\C & \SE{\C} \ar[l]_-{\Pro} \ar[r]^-{\K} & \C.\\
}
}
\end{equation}
Recall also that when $\C$ is pointed protomodular
\cite{BOURN:2000a}, for each object
$X$ in $\C$, the assignment of each object $B$
to the isomorphism class of split extensions $E$
with $\K(E)=X$ and $\Pro(E)=B$, determines a functor
$\SplExt(-,X):\C^{\op}\to \Set$ which assigns to each morphism
$p:E\to B$ the morphism $\SplExt(p,X)$ defined by \emph{pulling back along $p$}.
The category $\C$ is action representable in the sense of
\cite{BORCEUX_JANELIDZE_KELLY:2005a}
 when each of these functors is representable and
is weakly action representable \cite{JANELIDZE:2022} when for each $X$ in $\C$ there
exists a weak representation, that is there is a pair $(M,\mu)$
where $M$ is an object in $\C$ and $\mu: \SplExt(-,X) \to \hom(-,M)$
is a monomorphism. Note that when the functor $\SplExt(-,X)$ is
representable the representing object will be written $[X]$.

Action representability can also be rephrased as requiring
that for each $X$ in
$\C$ the fiber $\K^{-1}(X)$ has a terminal object, and action accessibility
\cite{BOURN_JANELIDZE:2009} can be phrased, by the weakening of this, to instead
require that for each $X$ in $\C$ the fiber $\K^{-1}(X)$ has \emph{enough
sub-terminal objects}, that is,
each object admits a morphism into a sub-terminal object (= an object admitting
at most one morphism into it). The sub-terminal objects in $\K^{-1}(X)$ are
called faithful extensions.

In \cite{JANELIDZE:2022}, G. Janelidze proved for a semi-abelian category (in
 the sense of Janelidze, Marki, Tholen \cite{JANELIDZE_MARKI_THOLEN:2002})
weakly-action-representability implies action accessibility (Theorem 4.6 of
\cite{JANELIDZE:2022}).
The main purpose of this paper is to show that the converse does not hold.
We show that a Birkoff subcategory of a (weakly)
action representable category is not necessarily weakly action representable.
This should be contrasted with the fact that a Birkoff sub-category of
action accessible category is necessarily action accessible
\cite{BOURN_JANELIDZE:2009}, and the immediate Proposition
\ref{proposition: normalizers restrict} below which
shows if $\C$ is a category admitting all normalizers (in the sense of \cite{GRAY:2013b} or
 in the sense of \cite{BOURN_GRAY:2015}), then every full
subcategory of $\C$ closed under sub-objects and finite limits, admits
all normalizers. Combining these two facts we show that the each category of
$n$-solvable groups ($n\geq 3$) is action accessible and has normalizers, but
is not weakly action representable.
\section{The results}
In this section we prove our main results.

Recall that the normalizer 
of a monomorphism $f:W\to X$ in \cite{GRAY:2013b} was defined to be the
universal factorization of $f$ as normal monomorphism followed
by a monomorphism
\[
\xymatrix{
W \ar@/_2ex/[rr]_{f} \ar[r]^{n} 
 & N\ar[r]^{m} &
X.
}
\]
A different definition was given in \cite{BOURN_GRAY:2015}, which in pointed, finitely complete contexted can be formulated as a commutative diagram
\[
\xymatrix{
W \ar@/_2ex/[dd]_{f} \ar[d]^{n} \ar[r]^{\kappa} & R \ar[d]^{\langle r_1,r_2\rangle}\\
N\ar[r]^-{\langle 0,1\rangle}\ar[d]^{m} & N\times N\\
X
}
\]
where the upper square is a pullback and the morphisms $r_1,r_2 : R\to N$ are the projections of an equivalence relation, which is universal amongst such commutative diagrams. The two definitions coincide in the pointed exact protomodular context, where $r_1,r_2 : R\to N$ is necessarily the kernel pair of its coequalizer, which in turn is necessarily a normal epimorphism with kernel $n$.

\begin{proposition}\label{proposition: normalizers restrict}
Let $\C$ be a pointed category admitting normalizers (in either sense). If $\X$ is
a full sub-category of $\C$ closed under subobjects and finite limits,
then $\X$ admits normalizers (in the same sense).
\end{proposition}
\begin{proof}
It is easy to check that under the conditions above the normalizer
in $\C$ of a monomorphism $f$ in $\X$ is also the normalizer of $f$
in $\X$.
\end{proof}
Recall that a span of monomorphisms $m:S\to B$ and $m':S\to B'$ (sometimes
called an amalgum) in a category $\C$ can be amalgamated in $\C$ if there
exist monomorphisms $u:B\to D$ and $u':B'\to D$ in $\C$ such that $mu=m'u'$.
\begin{proposition}
Let $\C$ be a action representable category, and let $\X$ be a Birkoff
subcategory of $\C$. The category $\X$ is not weakly action representable (and
hence not action representable), if there exist monomorphisms $m:S\to B$ and $m':S\to B'$ in $\X$, monomorphisms $u:B\to D$ and $u':B'\to D$ in $\C$, and $X$ in $\X$ with a monomorphism $v:D\to [X]$ in $\C$ such that
\begin{enumerate}[(i)]
\item  $um=u'm'$;
\item $m$ and $m'$ cannot be amalgamated in $\X$;
\item the split extensions corresponding to $vu$ and $vu'$ in $\C$ are
in $\X$.
\end{enumerate}
\end{proposition}
\begin{proof}
The monomorphisms $vu$, $vu'$, and $vum=vu'm'$ produce the span of
faithful extensions in $\X$. 
\[
\xymatrix{
X\ar[r]^{\kappa}&A\ar@<0.5ex>[r]^{\alpha}&B\ar@<0.5ex>[l]^{\beta}\\
X\ar@{=}[u]\ar@{=}[d]\ar[r]^{\lambda}&R\ar[d]_{g}\ar[u]^{f}\ar@<0.5ex>[r]^{\rho}&S\ar[d]^{m'}\ar[u]_{m}\ar@<0.5ex>[l]^{\sigma}\\
X\ar[r]^{\kappa'}&A'\ar@<0.5ex>[r]^{\alpha'}&B'\ar@<0.5ex>[l]^{\beta'}
}
\]
If $\X$ were weakly action representable, then $X$ would have weak representation
$M$ and there would (by Corollary 4.3 of \cite{JANELIDZE:2022}) be monomorphisms
$i:B\to M$ and $i':B'\to M$ in $\X$ such that $im =i'm'$. This is impossible
since $m$ and $m'$ can't be amalgamated in $\X$.
\end{proof}

\begin{example}
Let $\C$ be the category of groups.  Recall that: $\C$ is action representable \cite{BORCEUX_JANELIDZE_KELLY:2005a} with $[X]=\Aut(X)$ (the automorphism group of $X$),
$\C$ admits normalizers (in the sense of \cite{GRAY:2013b} or equivalently -- in this context --
in the sense of \cite{BOURN_GRAY:2015}), and amalgamation holds in $\C$ (which according to
\cite{KISS_MARKI_PROHLE_THOLEN:1982} was first proved in \cite{SCHREIER:1927}).
It is well-known that every group can be embedded in the automorphism group of an
abelian group (to prove this one can recall that every group can be embedded in the symmetric
group on its underlying set, and the symmetric group on a set $W$ can be embedded in
$\Aut(\Z_2^W)$). Now let $\X$ be the sub-variety of $n$-solvable groups ($n\geq 3$).
In \cite{NEUMANN_B_H:1959} B. H. Neumann has shown that there exists an abelian group $S$, a
$2$-nilpontent group $B$ and two monomorphisms $m:S\to B$ and $m':S\to B$
which can't be amalgamated as a solvable group. Since
$n$-nilpotent implies $n$-solvable, and split extensions with
kernel abelian and codomain $2$-solvable are at most $3$-solvable it follows
by the previous proposition that $\X$ is not weakly action representable (take
$D$ any group with $u:B\to D$ and $u':B'\to D$ monomorphisms 
such that $um=u'm'$ and then take $X=\Z_2^D$). It seems worth pointing out that
a finite such $D$ does exist (see e.g. Corollary 15.2 \cite{NEUMANN_B_H:1954}).
Action accessibility of $\X$ follows from Proposition 2.3 of
\cite{BOURN_JANELIDZE:2009} together with action accessibility of $\C$ (which
in turn follows immediately from $\C$ being action representable).
Note that $\X$ also has normalizers by
Proposition \ref{proposition: normalizers restrict}. Action
accessibility of $\X$ can then also be obtained from Proposition 4.5 of
\cite{BOURN_GRAY:2015} (see also \cite{GRAY:2015a} where it is proved that
action accessibility is equivalent to the existence of certain normalizers). 
\end{example}

\end{document}